\documentclass[12pt]{amsart}
\usepackage[T1]{fontenc}	    
\usepackage[utf8]{inputenc}         
\usepackage[english]{babel}
\usepackage{amscd}
\usepackage{amsmath,amsfonts,amsthm,epsfig,latexsym,graphicx,amssymb,psfrag}
\usepackage{xspace}
\usepackage{hyperref} 
\usepackage{xcolor}
\usepackage{enumerate}
\usepackage{array} 
\usepackage{appendix}

\newtheorem*{claii}{Claim}
\newtheorem{coro}{Corollary}

\newtheorem{defi}{Definition }[section]

\newtheorem{lemm}{Lemma}[section]
\newtheorem{prop}{Proposition}[section]
\newtheorem{proper}{Property}
\newtheorem{propri}{Properties}[section] 
\newtheorem{theo}{Theorem} 

\theoremstyle{definition}
\newtheorem{rema}{Remark}

\def\sup{\mathop{\rm Sup}}

\newfont{\df}{cmssbx10} 

\newtheorem*{definition*}{Definition }

\newtheorem*{notation*}{Notation}

\newcommand{\iet}{\mathcal G}
\newcommand{\ietf}{\overline{\mathcal G}}
\newcommand{\rl}{\mathcal R}
\newcommand{\mo}{\text{mod }}
\newcommand{\disc}{\text{Disc}}
\newcommand{\id}{\text{Id}}
\newcommand{\p}{\mathcal P} 
\newcommand{\I}{\mathcal I}

\newcommand{\ie}{\textsc{iet }}

\title[Uniform perfectness of IET and FIET]{Uniform perfectness for Interval Exchange Transformations with or without Flips.}

\author{ Nancy Guelman and Isabelle Liousse  }
\begin{document}

\address{{\bf  Nancy  GUELMAN}, IMERL, Facultad de Ingenier\'{\i}a, {Universidad de la Rep\'ublica, C.C. 30, Montevideo, Uruguay.} \emph {nguelman@fing.edu.uy}.}

\address{{\bf Isabelle LIOUSSE}, Univ. Lille, CNRS, UMR 8524 - Laboratoire Paul Painlevé, F-59000 Lille, France. \emph {isabelle.liousse@univ-lille.fr}.}

\begin{abstract}
Let $\mathcal G$ be the group of all  Interval Exchange Transformations. Results of Arnoux-Fathi (\cite{Ar1}), Sah (\cite{Sah})  and Vorobets (\cite{Vo}) state that $\mathcal G_0$  the subgroup of $\mathcal G$ generated by its commutators is simple. In (\cite{Ar1}), Arnoux proved that the group $\overline{\mathcal G}$ of all Interval Exchange Transformations with flips is simple.
 
We establish that every element of $\overline{\mathcal G}$ has a commutator length not exceeding $6$. Moreover, we give conditions on $\mathcal G$ that guarantee that the commutator lengths of the elements of $\mathcal G_0$ are uniformly bounded, and in this case for any $g\in \mathcal G_0$ this length is at most $5$. 

{\textbf{As analogous arguments work for the involution length in $\ietf$, we add an appendix whose purpose is to prove that every element of $\overline{\mathcal G}$ has an involution length not exceeding $12$.}}\end{abstract}

\maketitle
\section{Introduction.}

Let   $J=[a,b)$  be  a half-open interval.

\smallskip

An {\bf interval exchange transformation (IET)} of $J$ is a right continuous bijective map $f: J \rightarrow J$ defined by a finite partition of $J$ into half-open subintervals $I_i$ and a  reordering of these intervals by translations. We denote by $\mathcal G_J$ the group consisting in all IET of $J$.

\smallskip

An \textbf{interval exchange transformations with flips (FIET)} on $J$ is a bijection $f : J\rightarrow J$ for which there exists a subdivision $a = a_1 < \cdots < a_m <a_{m+1} = b$ such that $f\vert_{(a_{i},a_{i+1})}$ is a continuous isometry.  
Note that $f$ is not necessarily  right continuous since the orientation of some intervals can be reversed, and there exists a \textbf{flip-vector} {\footnotesize{$U(f)=(u_1,...,u_m)\in \{-1,1\}^n$}} such that $f\vert_{(a_{i},a_{i+1})}$ is 
a direct isometry if $u_i=1$ or an indirect one if $u_i=-1$.

It is worth noting that several authors consider FIET as piecewise isometries of $J$ with a finite number of discontinuity points reversing at least one of the intervals of continuity. This is not our case, so an IET is an FIET.

\smallskip

{\hskip 3 truecm\unitlength=0,16 mm\begin{picture}(210,210)
\put(0,0){\line(1,0){200}} \put(0,200){\line(1,0){200}}
\put(0,0){\line(0,1){200}} \put(200,0){\line(0,1){200}}
\put(50,0){\dashbox(0,200){}} \put(100,0){\dashbox(0,200){}} \put(150,0){\dashbox(0,200){}}
\put(0,50){\dashbox(200,0){}} \put(0,100){\dashbox(200,0){}} \put(0,150){\dashbox(200,0){}}
\put(0,0){\line(1,1){50}}
\put(150,150){\line(1,1){50}}
\put(50,100){\line(1,-1){50}}
\put(100,150){\line(1,-1){50}}
\put(-2,-2){$\scriptscriptstyle\bullet$} \put(48,98){$\scriptscriptstyle\bullet$}
\put(98,48){$\scriptscriptstyle\bullet$}  \put(148,148){$\scriptscriptstyle\bullet$}
\end{picture} } \hskip 1truecm
{\unitlength=0,16 mm\begin{picture}(210,210)
\put(0,0){\line(1,0){200}} \put(0,200){\line(1,0){200}}
\put(0,0){\line(0,1){200}} \put(200,0){\line(0,1){200}}
\put(50,0){\dashbox(0,200){}} \put(100,0){\dashbox(0,200){}} \put(150,0){\dashbox(0,200){}}
\put(0,50){\dashbox(200,0){}} \put(0,100){\dashbox(200,0){}} \put(0,150){\dashbox(200,0){}}
\put(0,0){\line(1,1){50}}
\put(150,150){\line(1,1){50}}
\put(50,100){\line(1,-1){50}}
\put(100,150){\line(1,-1){50}}
\put(-2,-2){$\scriptscriptstyle\bullet$} \put(98,148){$\scriptscriptstyle\bullet$}
\put(48,48){$\scriptscriptstyle\bullet$}  \put(148,98){$\scriptscriptstyle\bullet$}
\end{picture}}

\hskip 4.7 truecm \textit{Figure 1. Two equivalent FIET.} 

\smallskip

We say that two FIET on $J$, $f$ and $g$ are equivalent 
if the set $\{ x \in J \ : \  f(x)\not=g(x)\}$ is finite (see Figure 1). We denote by $\overline{\mathcal G}_J$ the corresponding quotient set.
Note that $\overline{\mathcal G}_J$ is the quotient group of the FIET group by its normal subgroup consisting of elements which are trivial except possibly at finitely many points. By abuse of terminology, elements of $\overline{\mathcal G}_J$ are also called FIET. The map $U$ is still well defined on $\overline{\mathcal G}_J$ and the set of all elements of $\overline{\mathcal G}_J$ such that $U(f)\subset \{1\}^m$ is identified with $\mathcal G_J$.

\medskip

Let $f$ be an IET or an FIET on $J$. The \textbf{continuity intervals} of $f$ are the maximal connected subsets of $J$ on which $f$ is continuous and they are denoted by $I_1, \cdots, I_m$. The \textbf{associated permutation} $\pi=\pi(f)\in \mathcal S_m$ is defined by $f(I_i)=J{\pi(i)}$, where the $J_j$'s are the ordered images of the $I_i$'s. By convention, $f$ is not continuous at the left endpoint of $J$ and we define $BP(f)$ to be the set of the discontinuity points of $f$. Note that $BP(f^{-1})=f(BP(f))$ and $BP( f \circ g) \subset BP(g) \cup g^{-1}(BP(f)).$ 

\medskip
Let $m$ be a positive integer, we denote by $\mathcal G_m$ [resp. $\mathcal G_{m,\pi}$] the set of all elements of $\mathcal G$ having at most $m$ discontinuity points [resp. whose associated permutation is $\pi\in \mathcal S_m$].

\medskip

From now on, without mention of the defining interval $J$, an IET or an FIET is defined on $I=[0,1)$ and  $\mathcal G_I$ and  $\overline{\mathcal G }_I$ will be denoted by $\mathcal G$ and $\overline{\mathcal G }$ respectively.

\begin{rema} \label{rem1} The group $\mathcal G_J$ [resp. $\overline{\mathcal G}_J$] is conjugated by the direct homothecy that sends $J$ to $I$ to the group $\mathcal G$ [resp. $\overline{\mathcal G}$]. The subgroup of $\mathcal G$ [resp. $\overline{\mathcal G}$] consisting of elements with support in $J$ can be identified, by taking restriction, with $\mathcal G_J$ [resp. $\overline{\mathcal G}_J$]. 
\end{rema}

Since the late seventies, the dynamics and the ergodic properties of a single interval exchange transformation were intensively studied (see e.g. Viana survey \cite{Vi}). A natural extension is to consider the dynamics in terms of group actions. The most famous problem was raised by Katok: does $\mathcal G$ contain copies of $\mathbb F_2$, the free group of rank $2$~? Dahmani, Fujiwara and Guirardel established that such subgroups are rare (\cite{DFG1} Theorem 5.2). 

More generally, one can ask for a description of possible subgroups of $\mathcal G$. According to Novak \textit{there is no distortion in $\mathcal G$} (\cite{No} Theorem 1.3) and as a standard consequence \textit{any finitely generated nilpotent subgroup of $\mathcal G$ is virtually abelian} (see e.g. \cite{GLIET}). 

Among many things, Dahmani, Fujiwara and Guirardel proved that \textit{any finitely generated subgroup of $\mathcal G$ is residually finite} (\cite{DFG1} Theorem 7.1), \textit{$\mathcal G$ contains no infinite Kazhdan groups} (\cite{DFG1} Theorem 6.2), \textit{any finitely generated torsion free solvable subgroup of $\mathcal G$ is virtually abelian} (\cite{DFG} Theorem 3)  and provide examples of non virtually abelian solvable subgroups of $\mathcal G$ (\cite{DFG} Theorem 6). Thus, finding torsion free finitely generated non virtually abelian subgroups of $\mathcal G$ seems very difficult, especially as  works of Juchenko, Monod (\cite{JM1}) suggest that $\mathcal G$ could be amenable as it is conjectured by Cornulier. 

\smallskip

The group $\mathcal G$ shares many of the properties of the group of piecewise affine increasing homeomorphisms of the unit interval, $PL^+ (I)$. For instance these two groups are not simple but have simple derivated subgroups (see \cite{Ep} for the PL case). As noted in Remark 1.2 of \cite{No2}, they satisfy no law (i.e. there does not exist $\omega \in \mathbb F_2\setminus\{e\}$ such that $\phi(\omega)=Id$ for every homomorphism $\phi : \mathbb F_2 \rightarrow G$), their nilpotent subgroups are virtually abelian (see \cite{FaFr} for the PL case) and the main result of \cite{DFG1} can be seen has a generic version of the Brin and Squier theorem \cite{BS} which asserts that $PL^+ (I)$ does not contain non abelian free subgroups. It's however still unknown whether $PL^+(I)$ is amenable. 

\smallskip

Considerably less is known about $\overline{\mathcal G}$. However, due to their connections with non oriented measured foliations on surfaces and billiards, the dynamics and ergodic properties of a single FIET were firstly explored by Gutierrez (\cite{Gu}) and Arnoux (\cite{Ar1}).

\smallskip

Dealing with irreducible permutations, Keane (\cite{Ke}) proved that almost all IET are minimal and Masur (\cite{Mas}) and Veech (\cite{Vee}) that almost all are uniquely ergodic. For FIET that reverse orientation in at least one interval, Nogueira (\cite{Nog}) proved that almost all have periodic points so are nonergodic. He also exhibited an example of a minimal uniquely ergodic one (see also \cite{Gutetal} and  \cite{LinSol}). Recently, Skripchenko and Troubetzkoy gave bounds for the Hausdorff dimension of the set of minimal maps (\cite{SkrTro}) and Hubert and Paris-Romaskevich described all the minimal maps having $4$ continuity  intervals (see \cite{HR} Theorem 6 p 39-40).

\smallskip

The decomposition into minimal and periodic components was first studied for measured surfaces flows by Mayer (\cite{May}) and restated for IET by Arnoux (\cite{Ar}) 
and Keane (\cite{Ke}). For FIET that reverse orientation in at least one interval of the $m$ continuity intervals, Nogueira, Pires and Troubetzkoy proved that the sum of number of periodic components and twice the number of minimal components is bounded by $m$ (\cite{NogPirTro}). 

As mentioned by Paris-Romaskevich in \cite{PaRo}, {\it one can interest ourselves in the dynamics of FIET from the point of view of geometrical group theory}: describe the possible groups that can be realized as groups of FIET or establish algebraic properties of the whole group. 

\smallskip

Here we shall be concerned solely with the structure of the whole groups $\mathcal G$ and $\overline{\mathcal G}$. This is also motivated by the algebraic study of other transformation groups, particularly groups of homeomorphisms of low dimensional manifolds, that was initiated by Schreier and Ulam in 1934 (\cite{SU}) who were interested in the simplicity of such groups.

\smallskip

We recall that, given $G$ a group,
\begin{itemize}
\item  a  \textbf{commutator} in $G$ is an element of $G$ of the form $[f,g]=fgf^{-1} g^{-1}$ with $f,g \in G$.
\item  $G$ is \textbf{perfect} if $ G=[G,G]$ the subgroup of $G$ generated by its commutators.
\item  $G$ is \textbf{simple} if any normal subgroup of $G$ is either $G$ or trivial.  
\end{itemize}

In the seventies, lots of homeomorphisms or diffeomorphisms groups were studied by Epstein, Herman, Thurston, Mather, Banyaga, and proved to be simple; these works are survey in the books \cite{Ba} or \cite{Bou}. For interval exchange transformations, it has been shown by Arnoux (\cite{Ar1} III \S 2.4), Sah (\cite{Sah}) and Vorobets (\cite{Vo}) that the subgroup $\mathcal G_0$ of $\mathcal G$ generated by its commutators is simple. In \cite{Ar1} (III \S 1.4), Arnoux  proved that $\overline{\mathcal G}$ is simple, this unpublished result has been recently recovered by Lacourte (\cite{Lac1}). In order to sharpen this property, it is convenient to give
\begin{defi}
Let $G$ be a group and $g\in  [G,G] $, the \textbf{commutator length} of $g$, denoted by $c(g)$, is the least number $c$ such that $g$ is a product of $c$ commutators. We set $c(G)= \sup\{ c(g), g \in [G,G]\}$ and we say that $G$ is \textbf{uniformly perfect} if $c(G)$ is finite.
\end{defi}

The main theorems of this paper are
\begin{theo}\label{thG} $\displaystyle c(\overline{\mathcal G})\leq 6.$\end{theo}
For the group $\mathcal{G}$ we are not able to decide if  $c(\mathcal{G})$ is finite.
However, in the affirmative case, we give an explicit bound in the following 
\begin{theo}\label{th3} \ 
If  $\mathcal G$ is uniformly perfect then $c(\mathcal G) \leq 5$.
\end{theo}

In section 6, we will prove stronger results, Theorems \ref{thGm} and \ref{th4}, which only require that commutator lengths are bounded when prescribing the number of discontinuity points or the arithmetic, that is when the elements considered belong to
$\Gamma_{\alpha}:=\{ g \in \mathcal{G} : BP(g) \subset \Delta_{\alpha}\}$, where $\Delta_{\alpha}$ is the abelian subgroup of $\mathbb R$ generated by $p$ real numbers $\alpha_1,\cdots,\alpha_p$ and $1$.

Our proofs are based on an adaptation of a result of Dennis and Vaserstein giving a criterion for uniform perfectness (\cite{Va}). This is explained in Section 4.

\smallskip

It is plain that a simple group $G$ is generated by any subset $S$ which is invariant under conjugation. In particular, $S$ can be the set consisting of commutators, involutions or finite order elements. Group invariants are therefore provided by considering $S$-lengths that is the least number $l_S$ such that any element can be written as product of $l_S$ elements of $S$. These lengths can be simultaneously considered by using the following
\begin{defi}
A group $G$ is \textbf{uniformly simple} if there exists a positive integer $N$ such that for any $f,g \in G\setminus \{id\}$, the element $g$ can be written as a product of at most $N$ conjugates of $f$ or $f^{-1}$.
\end{defi} 

Ulam and Von Neumann (\cite{UvV}) showed that the group of homeomorphisms of $\mathbb S^1$ is uniformly simple. Burago and Ivanov (\cite{BuIv}) obtained implicitely the same conclusion for $PL^+ (\mathbb S^1)$ and $[PL^+ (I),PL^+ (I)]$ (see also \cite{GaGi} Theorem 1.1). The question of the uniform simplicity of $[\mathcal{G},\mathcal{G}]$ and $\overline {\mathcal{G}}$ is formulated in \cite{PaRo}. However, Cornulier communicated us that $[\mathcal{G},\mathcal{G}]$ and $\overline {\mathcal{G}}$ are not uniformly simple. Indeed, if the support of an IET or FIET $f$ has length less than $\frac 1 N$ then any product of $N$ conjugates of $f$ or $f^{-1}$ can not have full support.

\medskip

In a forthcoming paper we will prove the uniform simplicity of $\mathcal A$, the group of affine interval exchange transformations of $I$, i.e. {\footnotesize{bijections $I \rightarrow I$ defined by a finite partition of $I$ into half-open subintervals such that the restriction to each of these intervals is a direct affine map}}. More generally, we will give conditions on groups of piecewise continuous bijective maps on $I$ that ensure uniform simplicity. This study and the work of \cite{GaGi} suggest that most simple transformation groups are uniformly simple. The group $\overline{\mathcal{G}}$ provides an example of a non uniformly simple group with bounded commutator length. As far as we know it is an open problem to determine whether $\overline{\mathcal{G}}$ has bounded involution length. Recently, O. Lacourte (\cite{Lac2}) defined the analogues of $\Gamma_{\alpha}$ in $\overline{\mathcal{G}}$, namely $\overline{\Gamma_{\alpha}}$. He proved that $[\overline{\Gamma_{\alpha}}, \overline{\Gamma_{\alpha}}]$ are simple and it would be relevant to study their uniform perfectness.

\medskip

\noindent \textbf{Acknowledgements.} We thank E. Ghys for communicating us, a long time ago, how  Proposition 1 of \cite{Va} allows to conclude that certain groups of homeomorphisms are uniformly perfect. We thank Y. Cornulier for fruitful discussions. We acknowledge support from the MathAmSud Project GDG 18-MATH-08, the Labex CEMPI (ANR-11-LABX-0007-01), the University of Lille (BQR), the I.F.U.M. and the project ANR Gromeov (ANR-19-CE40-0007). The second author also thanks CNRS for the  d\'el\'egation during the academic year 2019/20.
\section{Preliminaries.}
The aim of this section is to fix notations and terminology, to collect a few results and to prove some basic results to be used in the sequel.
\subsection{Restricted rotations and periodic IET}
\begin{defi} \ 

\smallskip

An IET with two continuity intervals is called a {\bf rotation} and it is denoted by $R_a$, where $a$ is the image of $0$.

An IET $g$ whose support, $supp(g)=\{x\in I : g(x)\not=x\}$, is  $J=[a,b)\subset[0,1)$ 
is a {\bf restricted rotation} if the direct homothecy that sends $J$ to $[0,1)$ conjugates $g_{\vert J}$ to a rotation. We denote it by $R_{\alpha,J}$ where $\alpha$ is given by $R_{\alpha,J}(x) =x+\alpha \ (mod \ \vert b-a \vert)$ for $x\in J$.

\smallskip

An element $g$ of $\mathcal G$ [resp. $\overline{\mathcal G}$] is \textbf{periodic} if every $g$-orbit is finite.

\end{defi}

\medskip

By \cite{Ar1} (III p.3), \cite{No} (Lemma 6.5) or \cite{Vo} (Lemma 2.1), any interval exchange transformation is a product of restricted rotations (see also our Lemma \ref{invRR} for a proof). For periodic IET, Novak showed a sharper statement.
\begin{lemm}\label{clcoinG}(\cite{No} Proof of Corollary 5.6) Any periodic element $g$ of $\mathcal G$ is conjugated in $\mathcal G$ to a product of finite order restricted rotations with disjoint supports. In particular, any periodic IET has finite order.
\end{lemm}

\subsection{Basic properties on commutators.} \ 
\begin{defi}Let $G$ be a group. An element $a\in G$ is \textbf{reversible} in $ G$ if there exists $h\in G$ such that $a=h a ^{-1} h ^{-1}$.\end{defi}
\begin{propri} \label{PrtyCom}Let $G$ be a group and let $a, b, a', b'$ and $h$  be elements of $G$.
\begin{enumerate}
\item If $a$ and $b$ commute with both $a'$ and $b'$ then $[a,b][a',b']= [aa',bb']$.
\item If $a'=h a ^{-1} h ^{-1}$  then $a a'=[a,h]$.
\item If $a$ is reversible in $ G$ then $a^2$ is a commutator.
\item $h [a,b] h ^{-1} = [h a h ^{-1}, h b h ^{-1}]$.
\end{enumerate}
\end{propri}
\noindent \textit{Proof.} 
\begin{quote} \vskip -6truemm 
\begin{enumerate} \item As  $a'$ commutes with $a$ and $b$, we have

\smallskip

\quad $[a,b][a',b']= ab a^{-1} b^{-1} a' b' a'^{-1} b'^{-1} = a a' b a^{-1} b^{-1}   b' a'^{-1}b'^{-1}.$

\smallskip
 
Repeating this process with  $b'$ and then $a'^{-1}$, we get 

\smallskip

\quad $[a,b][a',b'] = a a' b b' a^{-1} a'^{-1} b^{-1} b'^{-1}=a a' b b' a'^{-1} a^{-1} b'^{-1}  b^{-1}= [aa',bb'].$

\smallskip

\item  $a a' = a h a^{-1} h ^{-1} =[a,h]$.
\item $a^2 = a h a^{-1} h ^{-1} =[a,h]$.
\item $ [h a h ^{-1}, h b h ^{-1}]= h a h ^{-1} h b h ^{-1}h a^{-1} h ^{-1} h b^{-1} h ^{-1}= h [a,b] h ^{-1}.$ \hfill $\square$
\end{enumerate} 
\end{quote}
\subsection{Periodic IET are commutators.} \ 

\smallskip

In \cite{GLIET} Theorem 4, the authors proved that any periodic IET is reversible in  $\mathcal G$. We recall this argument briefly. Let $f$ be a periodic IET, by the Arnoux Decomposition Theorem (see Proposition page 20 of [Arn81a]), the interval $[0,1)$ can be written as the union of finitely many $f$-periodic components $M_i$, $i=1,..,n$ of period $p_i$. In particular, $M_i=\sqcup_{k=1}^{p_i} J_k$, where $J_k= f^{k-1}(J_1)$ are half-open intervals and $f$ is continuous on $J_k$.  

Eventually conjugating $f$ by an IET, we can suppose that the $J_k$'s are ordered consecutive intervals so the $M_i$'s are intervals and $\pi=\pi(f\vert_{M_i})=(1,2,\cdots, p_i)$.

We consider the IET $h$, that is defined on each $M_i$ by $h$ is continuous on $J_k$ and $h(J_k)= J_{\tau(k)}$, where $\tau \in \mathcal S_{p_i}$ satisfies $\tau^{-1} \pi \tau (k)=\pi^{-1}$ (such a permutation exists by Proposition 3.4 of \cite{FS}). One has that $h^{-1}\circ f \circ h$ is continuous on $J_k$ and $h^{-1}\circ f \circ h(J_k) =J_{\tau^{-1} \pi \tau (k)}=J_{\pi^{-1}(k)}$. Therefore $h^{-1}\circ f \circ h=f^{-1}$, meaning that $f$ is reversible in $\mathcal G$.

\smallskip

This implies \begin{prop}\label{th2} Any periodic IET is a commutator in $\mathcal G$. \end{prop} 
Indeed, we claim that any periodic IET can be written as the square of another periodic element. To see this, it is enough to consider rotations, by Lemma \ref{clcoinG}. This is obvious since $R_{\alpha}=R_{\frac{\alpha}{2}}^2$, so any periodic IET is the square of a reversible IET. Finally, the result follows from Properties \ref{PrtyCom} (3).

\section{Generalities on commutators in $\overline{\mathcal G}$ and ${\mathcal G}$.} 
\subsection{Commutators in $\overline{\mathcal{G}}$.}
\subsubsection{Fundamental examples.} {Let $a,b$ satisfy that $0\leq  a < b\leq 1$.}

\smallskip

We denote by  $\mathcal I_{[a,b]}$ the symmetry of  $[a,b]$, i.e. the FIET defined by: 
$$\mathcal I_{[a,b]}(x)=x \ \text{ if } \ \ x\notin [a,b] \ \ \ \ \text{  and } \ \ \  \mathcal I_{[a,b]} (x)=a+b-x \ \text{ if } \  \ x\in [a,b].$$

Similarly, we denote by $\mathcal I_{(a,b)}$  the symmetry of $(a,b)$, i.e. the FIET defined by: 
$$\mathcal I_{(a,b)}(x)=x \ \text{ if } \ \ x\notin (a,b) \ \ \ \ \text{  and } \ \ \  \mathcal I_{(a,b)} (x)=a+b-x \ \text{ if } \  \ x\in (a,b).$$

\medskip

{\unitlength=0,35 mm
\hskip 0 truecm  \begin{picture}(95,95)
\put(0,10){\textcolor{gray}{\line(1,0){95}}}
\put(10,0){\textcolor{gray}{\line(0,1){95}}}
\put (20, 85){$ \mathcal I_{[a,b]}$}
\put(44, 8){$\scriptstyle \vert$}
\put(38,0){$\scriptstyle a$}
\put(69, 8){$\scriptstyle \vert$}
\put(72,0){$\scriptstyle b$}
\put(44,68){$\scriptscriptstyle\bullet$}
\put(68,44){$\scriptscriptstyle\bullet$}
\put(10,10){\line(1,1){35}}
\put(45,70){\line(1,-1){25}}
\put(70,70){\line(1,1){25}}
\end{picture}}
{\unitlength=0,35 mm
\hskip 1.5 truecm  \begin{picture}(95,95)
\put(0,10){\textcolor{gray}{\line(1,0){95}}}
\put(10,0){\textcolor{gray}{\line(0,1){95}}}
\put (20, 85){$ \mathcal I_{(a,b)}$}
\put(44, 8){$\scriptstyle \vert$}
\put(38,0){$\scriptstyle a$}
\put(69, 8){$\scriptstyle \vert$}
\put(72,0){$\scriptstyle b$}
\put(43,43){$\scriptscriptstyle\bullet$}
\put(68,68){$\scriptscriptstyle\bullet$}
\put(10,10){\line(1,1){35}}
\put(45,70){\line(1,-1){25}}
\put(70,70){\line(1,1){25}}
\end{picture}}
{\unitlength=0,35 mm
\hskip 2 truecm  
\begin{picture}(95,95)
\put(0,10){\textcolor{gray}{\line(1,0){95}}}
\put(10,0){\textcolor{gray}{\line(0,1){95}}}
\put(20,85){$S_{\theta, [a,b)}$}

\put(44, 8){$\scriptstyle \vert$}
\put(38,0){$\scriptstyle a$}
\put(54, 8){$\scriptstyle \vert$}
\put(52,0){$\scriptstyle \theta$}
\put(69, 8){$\scriptstyle \vert$}
\put(72,0){$\scriptstyle b$}
\put(43,54){$\scriptscriptstyle\bullet$}
\put(53,44){$\scriptscriptstyle\bullet$}
\put(69,69){$\scriptscriptstyle\bullet$}
\put(10,10){\line(1,1){35}}
\put(55,45){\line(-1,1){11}}
\put(70,55){\line(-1,1){15}}
\put(70,70){\line(1,1){25}}
\end{picture}}

\medskip
Clearly, $\mathcal I_{[a,b]}$ and $\mathcal I_{(a,b)}$ are involutions and they represent the same element of $\overline{\mathcal G}$. Therefore, given  $J$ a subinterval of $I$, we define $\mathcal I_{J}$ to be the element of $\overline{\mathcal G}$ represented by $\mathcal I_{\bar J}$.

\smallskip

Let $\theta \in [a,b)$, we define another involution $S_{\theta, [a,b)}$ on $[a,b)$ by
$$S_{\theta, [a,b)}= \mathcal I_{[a,\theta]} \circ \mathcal I_{(\theta,b)}.$$

In particular, $S_{\theta, [0,1)}= \theta -x \ \ (mod \ 1)$ and it is denoted by $S_\theta$.

\begin{proper} \label{Sym} \ 

\begin{enumerate}
\item $S_\theta \circ S_{ \theta'} =R_{\theta-\theta'}$.
\item $R_{\alpha} \circ S_{\theta} \circ R_{\alpha}^{-1} = S_{\theta + 2\alpha}$.
\end{enumerate}
\end{proper}

\begin{lemm} \label{IRCom}(\cite{Ar1} III p.3) \ 

The maps $\mathcal I_{(a,b)}$ and $R_{\alpha,[a,b)}$ are commutators in $\overline{\mathcal G}_{[a,b)}$ and then in $\overline{\mathcal G}$.
\end{lemm}

\begin{proof} Taking restrictions and conjugating by a homothecy as in Remark \ref{rem1}, it is sufficient to prove that $\mathcal I_{(0,1)}$ et $R_{\alpha,[0,1)}$ are commutators.

It is easy to see that $\mathcal I_{(0,1)}$ is the product of the involutions $f_1$ and $f_2$ described as below: 

{\unitlength=0,3 mm
 \hskip 2.5 truecm \begin{picture}(120,120)
\put (20,98){$f_1$}
\put(0,10){\textcolor{gray}{\line(1,0){100}}} 
\put(10,0){\textcolor{gray}{\line(0,1){100}}}

\put(30, 8){$\scriptstyle \vert$}
\put(26,0){$\scriptstyle \frac{1}{4}$}
\put(70, 8){$\scriptstyle \vert$}
\put(71,0){$\scriptstyle \frac{3}{4}$}

\put(10,10){\line(1,1){20}}
\put(30,70){\line(1,-1){40}}
\put(70,70){\line(1,1){20}}
\end{picture} \hskip 3 truecm
\begin{picture}(120,120) 
\put (20,98){$ f_2$}
\put(0,10){\textcolor{gray}{\line(1,0){100}}}
\put(10,0){\textcolor{gray}{\line(0,1){100}}}

\put(30, 8){$\scriptstyle \vert$}
\put(26,0){$\scriptstyle \frac{1}{4}$}
\put(70, 8){$\scriptstyle \vert$}
\put(71,0){$\scriptstyle \frac{3}{4}$}

\put(10,90){\line(1,-1){20}}
\put(30,30){\line(1,1){40}}
\put(70,30){\line(1,-1){20}}
\end{picture}}

\smallskip

As $f_2$ is conjugated to $f_1=f_1^{-1}$ by  $R_{\frac{1}{2}}$, Item (2) of Properties \ref{PrtyCom} implies that the map $\mathcal I_{(0,1)}$ is a commutator.

\medskip

According to Property \ref{Sym}, any rotation is the product of $2$ symmetries that are conjugated by a rotation; thus  $R_{\alpha,[0,1)}$ is a commutator. \end{proof}

\subsubsection{Decomposition in involutions and restricted rotations.} 
\begin{lemm} \label{invRR} (\cite{Ar1} III p.3.) 
\begin{enumerate}
\item Any $f\in \overline{\mathcal G}$ can be written as the product of an element of ${\mathcal G}$ and an involution that is a commutator.
\item Any $g\in \mathcal G_m$ can be written as the product of $m-1$ restricted rotations.
\end{enumerate}
\end{lemm}

\begin{proof} (1) Let $f\in \overline{\mathcal G}$, we denote by $I_i$ the continuity intervals of $f$ and by  $U=(u_i)$ its flip-vector. It is easy to check that $f \circ \prod_{\{i \ \vert \  u_i=-1\}} \mathcal I_{I_{i}} $ belongs to ${\mathcal G}$. Moreover the $\mathcal I_{I_{i}}$'s have disjoint supports, so they commute and then $\prod_i \mathcal I_{I_{i}}$ is an involution and a commutator by Lemma \ref{IRCom} and Properties \ref{PrtyCom}.

\medskip

\qquad (2) For clarity, given $J=[a,b)$ and $K=[b,c)$ two consecutive half-open intervals, we denote by $R_{J,K}$ the restricted rotation of support $J\sqcup K$ whose interior discontinuity point is $b$.

Let $g\in \mathcal G_{m,\pi}$ with continuity intervals $I_1,\cdots I_m$ and let $g(I_i)=J_{\pi(i)}$. We consider $R_1= R_{J,K}$, where $J=J_1 \cup \cdots \cup J_{\pi(1)-1}$ and $K=J_{\pi(1)}$. One directly has that $R_1\circ g \vert _{I_1}=Id$ and $\#BP(g_1)\leq m-1$, where $g_1 = R_1\circ g \vert _{I_2 \cup \cdots \cup I_m}$.

Starting with $g_1$, we define similarly $R_2$ and we get that $R_2\circ g_1 \vert _{I_2}=Id$ and $\#BP(g_2)\leq m-2$, where $g_2 = R_2\circ g_1 \vert _{I_3 \cup \cdots \cup I_m}$.

\smallskip

Repeating the previous argument $m-1$ times leads to $\#BP(g_{m-1})\leq 1$ so $g_{m-1}=Id$.

\smallskip

Extending the restricted rotations $R_i$ to $[0,1[$ by the identity map, we conclude that $$R_{m-1} \circ \cdots \circ R_1 \circ g = Id.$$ \vskip -6truemm \end{proof}

A direct consequence of Lemmas \ref{invRR} and \ref{IRCom} is

\begin{prop} \label{perf} (\cite{Ar1} III \S 1.4)

The group $\overline{\mathcal G}$ is perfect and any $g\in\mathcal G_m$ is the product of $m-1$ commutators in $\overline{\mathcal G}$.
\end{prop}

\subsection{Commutators in $\mathcal{G}$.} \ 

In the introduction, we have indicated a few similarities between the groups ${\mathcal G}$ and $PL^+(I)$.
In particular, the simplicity of their derivative subgroups relies on a result of Epstein (see  \textsc{1.1.Theorem} in \cite{Ep}). In the context of finding bounds for the commutator length, a substantial difference between these two groups is that an element $f$ of $[PL^+(I),PL^+(I)]$ is a map whose support $J$ satisfies  $\bar J \subset (0,1)$ and $f$ can not be written as a product of commutators of maps with support in $J$. This contrasts with

\begin{rema} \label{remSAF} Let $J$ be a half-open subinterval of $I$.

If $g\in\mathcal G$ has support in $J$ then $g\in [\mathcal G,\mathcal G]$ if and only if $g\vert_J\in [\mathcal G_J,\mathcal G_J]$. \end{rema}

Indeed, according to Theorem 1.1 of \cite{Vo}, there is a morphism $SAF_J \ : \mathcal G_J\rightarrow \mathbb R  \otimes_{\mathbb Q} \mathbb R$  such that $[\mathcal G_J,\mathcal G_J] = SAF_J^{-1}(0)$. More precisely for $f\in \mathcal G_J$, $ SAF_J(f)=\sum  \lambda_k \otimes \delta_k,$ where the vectors $(\lambda_k)$, $(\delta_k)$ encode the lengths of exchanged intervals and  the corresponding translation constants respectively.

Let $g\in \mathcal G$ with support in $J$. From the previous definition, it is easy to check that $SAF_I(g) = SAF_J (g\vert_J)$. Therefore $g\in [\mathcal G,\mathcal G]$ if and only if $g\vert_J\in [\mathcal G_J,\mathcal G_J]$.
\section{The adapted Dennis and Vaserstein argument.} 
In this section, we first we recall Proposition 1(c) of Dennis and Vaserstein (\cite{Va}).
\subsection{The original criterion.} 
\begin{defi} Two subsets $S_1$ and $S_2$ of a group $G$ are \textbf{commuting} if any $a\in S_1$ commutes with any  $a'\in S_2$. \end{defi}

\noindent \textbf{Dennis and Vaserstein's criterion.} \textit{If a group $G$ contains two commuting subgroups $H_1$ and $H_2$ such that for each finite subset $S$ of $G$ there are elements $g_i\in G$, $i=1,2$, such that $g\sb i\sp{-1}Sg\sb i\le H\sb i$ for $i=1,2$, then  $c(G)\le 3$.}

\smallskip

\noindent  As an illustrating example indicated by Ghys, the group $[PL^+(I),PL^+(I)]$ consists in all $g$ of $PL^+(I)$ such that $g'(0)=g'(1)=1$. Thus, for any finite collection $\{g_i\}$  in $[PL^+(I),PL^+(I)]$ there exist $0<a<b<1$ such that $(a,b)$ contains the support of all the $g_i$. The required groups $H_1$ and $H_2$ are obtained as groups of maps with disjoint supports by setting $H_1=\langle g_i \rangle$ and $H_2=h \langle g_i \rangle h^{-1} $ where $h$ is an element of $[PL^+(I),PL^+(I)]$ that carries $(a,b)$ into $(\frac{a}{2},a)$.

\smallskip

Unfortunately, this argument doesn't apply immediately in $[\mathcal G,\mathcal G]$ and $\overline{\mathcal G}$. This is essentially due to  the facts that both  groups contain maps with full support and that if the length of the support of $g\in \overline{\mathcal G}$ exceeds $\frac{1}{2}$ then it is impossible to find a conjugate of $g$ in $\overline{\mathcal G}$ with a disjoint support. To avoid these difficulties, we will compose with suitable periodic maps to obtain an IET with arbitrary small support (see Propositions 5.1 and 6.1) and then we will apply the following iterated version of Dennis and Vaserstein's criterion.

\subsection{The iterated version.}

Let $n\in\mathbb N^*$, we denote by $H_n$ [resp. $\overline{H}_n$] the subgroup of $\mathcal G$ [resp. $\overline{\mathcal G}$]  consisting of elements whose support is included in $[ 1-\frac{1}{n} ,1)$. 
By Remark \ref{rem1}, $H_n$ [resp. $\overline{H}_n$] is identified with $\mathcal G_{[ 1-\frac{1}{n} ,1)}$ [resp. $\overline{\mathcal G}_{[ 1-\frac{1}{n} ,1)}$]. Moreover, Remark \ref{rem1} and Proposition \ref{perf} imply that 
$\overline{H}_n$ is perfect and Remark \ref{remSAF} leads to $H_n\cap[\mathcal G,\mathcal G]=[H_n,H_n]$.

\begin{lemm}\label{VaserOpti} 
$\displaystyle (a) \ \ \  {\text{ If }} g\in H_2\cap[\mathcal G,\mathcal G]=[H_2,H_2] \ {\text{ then  }}
\ \  c_{\mathcal G}(g) \leq  \frac {1}{2}c_{H_2}( g)  + \frac {3} {2}.$

\hskip 2.2truecm $\displaystyle (b) \ \ \  {\text{ If  }} g\in \overline{H_2} \  {\text{ then  }}  \ \ c_{\overline{\mathcal G}}(g) \leq  \frac {1}{2}c_{\overline{H_2}}( g)  + \frac {3} {2}.$
\end{lemm}

\begin{proof} Proofs of Items $(a)$ and $(b)$ are similar, changing $\mathcal G$ for $\overline{\mathcal G}$ and $H_2$ for $\overline{H}_2$, so we only prove Item $(a)$.

Let $g\in H_2\cap[\mathcal G,\mathcal G]$. 
 We write $c_{H_2}(g) = 2p - r$ with $p\in \mathbb N^*$ and $r=0,1$. Therefore $$g=(c_1 ... c_p) (c_{p+1} ... c_{2p}),$$ where $c_i=[a_i, b_i]$ with $a_i$, $b_i$  in $H_2$ and the last commutator $c_{2p}$ is eventually trivial.

\smallskip

Let $R$ be the rotation of angle $\frac{1}{2}$. We denote by  $f'=R \circ f \circ R^{-1}$.

Note that if $f,k\in H_2$ then $f$ and $k'$ have disjoint supports and they commute. We write
$$g=(c_1 ... c_p) (c'_{p+1} ... c'_{2p}) (c'_{p+1} ... c'_{2p})^{-1} (c_{p+1} ... c_{2p}) =
(c_1c'_{p+1}) \  ... \ (c_p c'_{2p}) \ \  C,$$ where $C=(c'_{p+1} ... \ c'_{2p})^{-1} (c_{p+1} ... \ c_{2p})$.

On one hand, by Properties \ref{PrtyCom} (1), we have that $c_ic'_{p+i}$, $i=1,..., p$, are commutators.

On the other hand, by Properties \ref{PrtyCom} (2), it holds that $C=(c'_{p+1} ... c'_{2p})^{-1} (c_{p+1} ... c_{2p})$  is a commutator since it is the product of $(c_{p+1} ... c_{2p})$ and the conjugate by $R$ of its inverse.

Finally, we have $c_{\mathcal G}(g)\leq p+1$, thus
$$2c_{\mathcal G}(g)\leq 2p+2 = c_{H_2}( g) +r+2 \leq  c_{H_2}( g) +3.$$\vskip -0.7truecm\end{proof}

Repeatedly applying Lemma \ref{VaserOpti}, we get
\begin{prop} \label{Vasit} Let $t\in \mathbb N^*$.

$\displaystyle (a) \ \ \  {\text{ If }} g\in H_{2^t}\cap[\mathcal G,\mathcal G] =[ H_{2^t}, H_{2^t}] \  {\text{ then  }} \ \  c_{\mathcal G} (g) < \frac{1}{2^t} c_{H_{2^{t}}} (g) + 3.$

$\displaystyle (b) \ \ \  {\text{ If  }} g\in \overline{H_{2^t}} \  {\text{ then  }} \ \   c_{\overline{\mathcal G}} (g) < \frac{1}{2^t} c_{\overline{H}_{2^{t}}} (g) + 3.$
\end{prop}

\begin{proof}
As noted earlier, we only prove Item $(a)$.

Let $t\in\mathbb N^{*}$ and $g\in H_{2^t}\cap[\mathcal G,\mathcal G]$. From Remark \ref{rem1} 
and  Lemma \ref{VaserOpti}, we obtain  $$ c_{H_{2^{t-1}}} (g) \leq \frac{1}{2} c_{ H_{2^{t}}} (g) + \frac{3}{2}.$$
It is easy to check by induction that for $s\in \{1, \cdots ,t\}$,  we have $$(E_s) \ \   c_{H_{2^{t-s}}} (g) \leq \frac{1}{2^s} c_{ H_{2^{t}}} (g) + 3 \sum_{j=1}^{s}\frac{1}{2^j}.$$
{\footnotesize Indeed, for $s=1$, $(E_1)$ is the first identity.}

{\footnotesize Fix $s\in \{1, \cdots ,t-1\}$ and suppose that $(E_s)$ holds. Then according to Remark \ref{rem1} 
and Lemma \ref{VaserOpti} }
{\footnotesize $$ C_{H_{2^{t-(s+1)}}} (g) \leq \frac{1}{2} C_{ H_{2^{t-s}}} (g) + \frac{3}{2}.$$}
{\footnotesize Thus, by induction hypothesis }
{\footnotesize $$ C_{H_{2^{t-(s+1)}}} (g) \leq \frac{1}{2} \left( \frac{1}{2^s} C_{ H_{2^{t}}} (g) + 3 \sum_{j=1}^{s}\frac{1}{2^j}\right) + \frac{3}{2}.$$} 
{\footnotesize Therefore}
{\footnotesize $$C_{H_{2^{t-(s+1)}}} (g) \leq \left( \frac{1}{2^{s+1}} C_{ H_{2^{t}}} (g) + 3 \sum_{j=1}^{s}\frac{1}{2^{j+1}} \right) + \frac{3}{2},$$} {\footnotesize which leads immediately to $(E_{s+1})$.}

\medskip

Finally, noting that $H_1=\mathcal G$ and $\sum_{j=1}^{t}\frac{1}{2^j}= 1-\frac{1}{2^t}<1$, the identity $(E_t)$ leads to $$c_{\mathcal G} (g) < \frac{1}{2^t} c_{H_{2^{t}}} (g) + 3.$$ \vskip -0.7truecm\end{proof}

\section{The group $\overline{\mathcal G}$ is uniformly perfect.} 

The aim of this section is to prove Theorem \ref{thG}. 

\subsection{Background material.} Let $g\in \mathcal G_m$.

\smallskip

\noindent The \textbf{combinatorial description} of $g$ is $(\lambda(g), \pi(g))$, where $\lambda(g)$ is an $m$-dimensional vector whose coordinates are  the lengths of $I_1, \cdots , I_m $, the continuity intervals of $g$ and $\pi(g)\in \mathcal S_m$ is the permutation on $\{1,\cdots,m\}$ that tells how the intervals $I_i$ are rearranged by $g$.

\smallskip

\noindent  We denote by $a_i(g)$ the discontinuity points of $g$. If $g$ is continuous on a half-open interval $J$, we define ${\mathbf{\delta_J}}(g):= g(x) -x$, for $x\in J$.

\smallskip

\noindent  \textbf{The translations of $g$} are $\delta_i(g):=\delta_{I_i}(g)$, $i=1,\cdots,m$.

\smallskip

Note that $a_i(g)$ and $\delta_i(g)$ are related to $(\lambda(g), \pi(g))$ by $$(*)   \ \ \ \ a_i(g)=\sum_{k=1}^{i-1} \lambda_k(g) \ \ \text{ and } \ \  \delta_i(g)= -\sum_{k=1}^{i-1}\lambda_k(g)+\sum_{k=1}^{\pi(i)-1}\lambda_{\pi^{-1}(k)}(g).$$

The map $g$ is said to be \textbf{rational} if all its discontinuity points are rational. It is easy to see that rational IET are periodic.  

\begin{defi} Let $m$ be a positive integer and $\pi\in\mathcal S _m$, we define a \textbf{metric} on $\mathcal G_{m,\pi}$ by $$d(f,g) = \sum_{i=1}^{m} \vert \lambda_i(f) - \lambda_i(g)\vert.$$ 
\end{defi}

\begin{propri} \label{Dist} Let $f$ and $g$ be elements of $\mathcal G_{m,\pi}$. Then 
\begin{itemize}
\item $d(f^{-1},g^{-1}) = d(f,g)$,
\item $\vert a_i(f) - a_i(g) \vert \leq d(f,g)$,
\item $\vert \delta_i(f) -\delta_i(g) \vert \leq 2d(f,g)$.
\end{itemize} \end{propri}

\begin{proof} \ 

$\bullet$ The first item is due to the fact that $\displaystyle \lambda_{\pi(i)}(f^{-1})= \lambda_i(f)$.

\smallskip

We deduce the remaining items from $(*)$, indeed

{\footnotesize {$\bullet$ $\displaystyle \vert a_i(f) -a_i (g)\vert = \vert \sum_{k=1}^{i-1} \lambda_k(f)-\lambda_k(g) \vert \leq d(f,g)$ and 

$\bullet$ $\displaystyle \vert \delta_i(f) -\delta_i(g) \vert = \vert \bigl(-\sum_{k=1}^{i-1}\lambda_k(f)+\sum_{k=1}^{\pi(i)-1}\lambda_{\pi^{-1}(k)}(f)\bigr) -\bigl(-\sum_{k=1}^{i-1}\lambda_k(g)+\sum_{k=1}^{\pi(i)-1}\lambda_{\pi^{-1}(k)}(g)\bigr) \vert \leq 2d(f,g)$.}} \end{proof}

\begin{lemm} \label{Lemsup} Let $g\in \mathcal G_m$ and let $l=\vert Fix(g)\vert$ be the Lebesgue measure of the fixed point set of $g$. Then, there exists $h\in \mathcal G_m$ such that $$Fix(h\circ g \circ h^{-1}) = [0, l).$$ In particular $\#BP (h\circ g \circ h^{-1})\leq 3m$. \end{lemm}

\begin{proof}
Denote by $F_1$, $F_3$, $\cdots$ , $F_{2p-1}$ the $p$ ordered connected components of $I\setminus Fix(g)$.

We write $F_i=[\alpha_i,\alpha_{i+1})$, for $i=2k-1$, $k=1,..., p$. Note that $\alpha_i\in BP(g)$.

Hence the connected components of $Fix(g)$ are the possibly empty intervals $F_0=[0,\alpha_1)$, $F_{2p}=[\alpha_{2p},1)$ and $$F_{2k}=[\alpha_{2k},\alpha_{2k+1}), \text{ for } k=1,..., p-1.$$

\smallskip

The required map $h$ is the IET whose combinatorial description is $(\lambda, \pi)$ 
$${\text{ with } }\left\{\begin{array}{c} \lambda_i=\vert F_i\vert, \ i=0,\cdots, 2p \  \ \ { \text{ and } } \ \ \pi \in \mathcal S(\{0,...,2p\}),   \cr \pi(0)=0,\ \  \pi(2k)=k { \text{ and } }\pi(2k-1)=k+p,  \ \ k=1,\cdots, p. \end{array}\right.$$ \ \ \ Finally we note that  $h\in \mathcal G_m$ since $BP(h)\subset \{\alpha_i\} \subset BP(g)$.\end{proof}

\subsection{Proof of Theorem \ref{thG}.} \ 
For proving Theorem \ref{thG}, we need the following

\begin{prop}\label{Cpf}
Let $n$ be a positive integer and let $f\in \mathcal G_m$. Then there exist two periodic elements $p, p' \in \mathcal{G}$ such that $$\vert supp( p \circ f \circ p')\vert \leq \frac{1}{n} {\text {\ \  and \ \  } } \#BP (p \circ f \circ p') \leq 5m.$$
\end{prop}

\begin{proof}

Let $n$ be a positive integer and $f\in \mathcal G_{m,\pi}$.  We set $BP(f)=\{ a_i, \ i=1 \cdots m \}$, $I_i=[a_i, a_{i+1})$ and $BP(f^{-1}) = \{  b_i, i=1,\cdots, m \}$. Fix $0<\epsilon < \frac{1}{2n}$ small enough ($\epsilon \ll \vert I_i \vert$). 

\medskip

We consider $p\in \mathcal G_{m,\pi^{-1}}$ a rational IET such that  $d(f^{-1}, p) \leq \frac{\epsilon}{2m}$ and  $BP(p)=\{b'_i, i=1,\cdots, m\}$ satisfies $b_i-\frac{\epsilon}{2m} < b'_i \leq  b_i$. This map $p$ is periodic.

\medskip

\begin{claii} \label{clpof} By construction, $f_{\epsilon}=p\circ f$ satisfies  $\#BP(f_{\epsilon})\leq 2m$, it is continuous on $[a_i,a_{i+1} -\frac{\epsilon}{2m})$ and $\partial_i:=\delta_{[a_i,a_{i+1} -\frac{\epsilon}{2m})}(p\circ f)$ satisfies $\vert \partial_i\vert \leq \frac{\epsilon}{m}$.\end{claii}

Indeed, obviously $\#BP(f_{\epsilon})\leq \#BP(f) + \#BP(p) \leq 2m$.

\smallskip

For every $x\in [a_i,a_{i+1} -\frac{\epsilon}{2m})$, one has $f(x) = x+\delta_{I_i} (f)$ and  
$$f(x) \in [b_{\pi(i)},b_{\pi(i)+1}-\frac{\epsilon}{2m}) \subset [b'_{\pi(i)},b'_{\pi(i)+1}), {\text { then }}$$
$$p\circ f (x) = x + \delta_{I_i} (f) +  \delta_{[b'_{\pi(i)},b'_{\pi(i)+1})} (p).$$

Since $d(f^{-1}, p) \leq \frac{\epsilon}{2m}$, one has: 
$$ \frac{\epsilon}{m}\geq \vert \delta_{[b'_{\pi(i)},b'_{\pi(i)+1})} (p) - \delta_{[b_{\pi(i)},b_{\pi(i)+1})} (f^{-1}) \vert =  \vert \delta_{[b'_{\pi(i)},b'_{\pi(i)+1})} (p) + \delta_{I_i} (f) \vert, {\text { therefore }} $$
$$\vert \partial_i\vert=\vert p\circ f (x) -x \vert =\vert \delta_{I_i} (f) + \delta_{[b'_{\pi(i)}),b'_{\pi(i)+1})} (p) \vert \leq \frac{\epsilon}{m}.$$ This ends the proof of the claim which is summarized by the following picture. 

\bigskip

{\unitlength=0,37 mm
\hskip 4.5 truecm  \begin{picture}(110,90)
\put(0,10){\line(1,0){100}}
\put(-20,60){\line(1,0){120}}
\put(0,90){\line(1,0){90}}

\put(0, 8){$\scriptstyle [ $}
\put(0, -2){$\scriptstyle a_i $}

\put(80, 8){$\scriptstyle [ $}
\put(60, -2){$\scriptstyle a_{i+1} -\frac{\epsilon}{2m} $}

\put(99,8){$\scriptstyle [$}
\put(96, -2){$\scriptstyle a_{i+1}  $}

\put(0, 58){$\scriptstyle [ $}
\put(0,  48){$\scriptstyle b_{\pi(i)} $}

\put(-20, 58){$\scriptstyle [ $}
\put(-20,  48){$\scriptstyle b'_{\pi(i)} $}

\put(80, 58){$\scriptstyle [ $}
\put(55, 48){$\scriptstyle b_{\pi(i)+1} -\frac{\epsilon}{2m} $}

\put(89,58){$\scriptstyle [$}
\put(90,63){$\swarrow$}
\put(100, 72){$\scriptstyle b'_{\pi(i)+1} $}
\put(88,59){\textcolor{red}{$\scriptscriptstyle \bullet$}}

\put(99,58){$\scriptstyle [$}
\put(106, 48){$\scriptstyle b_{\pi(i)+1} $}

\put(37,30){$f$}
\put(47,20){$\vector(0,1){25}$}  
\put(0, 88){$\scriptstyle [ $}
\put(89,88){$\scriptstyle [$}
\put(88,89){\textcolor{red}{$\scriptscriptstyle \bullet$}}
\put(37,70){$p$}
\put(47,65){$\vector(0,1){17}$}
\put(88,9){\textcolor{red}{$\scriptscriptstyle \bullet$}}
\end{picture}}

\bigskip

We turn now on to the proof of Proposition \ref{Cpf}. Let $i\in \{1,\cdots,m\}$.

\medskip

If $\partial_i= 0$, we set $R_i=Id$.

\medskip

In the case that $\partial_i> 0$, we define $R_i$ to be the finite order restricted rotation of support  $[\ a_i \ , \  a_i + r_i \partial_i\  )$ and of angle $\partial_i$, where $r_i$ is the greatest integer such that $a_i+r_i \partial_i \leq min \{ \  a_{i+1} -(\frac{\epsilon}{2m}-\partial_i) \ , \  a_{i+1} \ \}$.

\medskip

By definition, $R_i$ and $f_\epsilon$ coincide on $[a_i,a_{i} + (r_i-1) \partial_i)$ and $\left\vert [\ a_{i} +r_i \partial_i\ , \ a_{i+1} \ ) \right\vert \leq  \frac{\epsilon}{m}$.

\smallskip

\begin{quote}{\footnotesize{Indeed, $f_\epsilon$ is continuous  on $[a_i,a_{i} + (r_i-1) \partial_i)$, since $$a_{i} + (r_i-1) \partial_i =a_{i} + r_i\partial_i -\partial_i \leq a_{i+1} -(\frac{\epsilon}{2m}-\partial_i)  -\partial_i = a_{i+1} -\frac{\epsilon}{2m}.$$ 
In addition, by the maximality of $r_i$, either $a_i+ (r_i+1)\partial_i $ is greater  than $$ a_{i+1} -(\frac{\epsilon}{2m}-\partial_i)\ \ \text{  thus } \ \vert [a_{i} +r_i \partial_i,a_{i+1})\vert= a_{i+1} - (a_{i} +r_i \partial_i)<\partial_i+(\frac{\epsilon}{2m}-\partial_i)=\frac{\epsilon}{2m} \text{ \ \ \ or}$$  \hskip -3truemm $a_{i+1}$ \hskip 2.3truecm so  \ \ \ \  
 $\vert [a_{i} +r_i \partial_i,a_{i+1})\vert= a_{i+1} - (a_i+ r_i\partial_i) < \partial_i \leq \frac{\epsilon}{m}$.}}\end{quote}
\medskip

The same argument remains valid for negative $\partial_i$ by using non positive integers $r_i$.
\medskip

Finally, the map $g:=f _{\epsilon} \circ  \prod_{1}^m {R_i}^{-1} $ satisfies $\#BP(g)\leq 5m$ because $\#BP(R_i)\leq 3$.

\smallskip

Since $supp (R_i) \subset [a_i, a_{i+1})$, the supports of the $R_i$'s are disjoints and $p'=(\prod_{1}^m R_i)^{-1}$ is periodic and it is also a commutator in ${\mathcal{G}}$, according to Proposition \ref{th2}.

\smallskip

But \ $\displaystyle g_{\displaystyle \vert_{ {\displaystyle R_i (} [a_i ~,~a_{i} + (r_i-1) \partial_i){\displaystyle )}}} \ = \ Id$, \ therefore

\vskip -0.3truecm

$$ \ \ \ \ \ \ \ \ \ \  \vert supp(g) \vert \leq 1- \sum_{i=1}^m \vert [a_i,a_{i} + (r_i-1) \partial_i) \vert$$
$$\ \ \ \ \ \ \ \ \ \ \ \ \ \ \ \ \ \ \ \ \ \ \ \ \ \ \ \  \ \ \ \ \ \ \ \ \  \ \ \ \leq 1-\sum_{i=1}^m (\vert [a_i,a_{i+1})\vert -  (\partial_i +\frac{\epsilon}{m}))\leq 2 \epsilon \leq \frac{1}{n}.$$
\vskip -0.9truecm\end{proof}


\smallskip

We turn now on to the proof of Theorem \ref{thG}. 

\smallskip

We first consider an IET $f\in \mathcal G_{m,\pi}$ viewed as an element of ${\overline{\mathcal G}}$. Let $t \in \mathbb N^*$.

Applying Proposition \ref{Cpf} to $f$ and $n=2^t$, we get that there exist two periodic elements $p, p'\in \mathcal{G}$ such that the support of $g= p \circ f \circ p'\in \mathcal{G}_{5m} $ has measure less than or equal to $\frac{1}{2^t}$. 

By Lemma \ref{Lemsup}, the map $g$ is conjugated to an element $g'$ of $H_{2^{t}}$ for which $\#BP(g')\leq 15m$. Since $p$ and $p'$ are periodic and $g$ and $g'$ are conjugated, we have
$$c_{\overline{\mathcal G}} (f) \leq c_{\overline{\mathcal G}} (g) + 2 = c_{\overline{\mathcal G}} (g') + 2.$$
Then by  Proposition \ref{Vasit} (b),
$$c_{\overline{\mathcal G}} (f) < \frac{1}{2^t} c_{\overline{H}_{2^{t}}} (g') + 5.$$
As $\#BP(g'\vert_{[1-\frac{1}{2^t} , 1)}) \leq \#BP(g')$, 
Remark \ref{rem1} and Proposition \ref{perf} imply that $$c_{\overline{H}_{2^{t}}} (g') \leq 15m-1.$$
Finally, for any $t\in \mathbb N^*$ one has  $$c_{\overline{\mathcal G}}(f) < \frac{15m-1}{2^t} + 5$$
and choosing $t$ large enough, we obtain  $$c_{\overline{\mathcal G}}(f) \leq  5. $$

Thus we get $c_{\overline{\mathcal G}}(f) \leq  5$, for any $f\in \mathcal G$.

\medskip

For the general case, we consider $F\in \overline{\mathcal G}$. According to Lemma \ref{invRR}, the map $F$ can be decomposed as the product of an involution that is a commutator and an element of $\mathcal G$. Therefore, we have proved that $c_{\overline{\mathcal G}}(F) \leq 1+5=6$, for any $F\in \overline{\mathcal G}$. \hfill $\square$


\section{Conditions for uniform perfectness of ${\mathcal G}$.} 

In this section we give two sufficient conditions for $\mathcal G$ to be uniformly perfect.
  
\subsection{The commutator length is bounded when fixing the number of discontinuity points.}

\smallskip

We prove the following statement that directly implies Theorem \ref{th3}.

\begin{theo}\label{thGm}
If for any positive integer $m$, $C_m(\mathcal G) := sup \{ c_{\mathcal G} (g ) \ , \ g \in [\mathcal G, \mathcal G] \cap \mathcal G_m \}$ is finite, then $\mathcal G$ is uniformly perfect and $c(\mathcal G) \leq 5$. \end{theo}

\begin{proof} Let $f\in [\mathcal G, \mathcal G] \cap \mathcal{G}_m$ and $t\in \mathbb N$. Proposition \ref{Cpf} and Lemma \ref{Lemsup} with $n=2^t$ show that that there exist two periodic elements $p', p\in \mathcal{G}$ such that $g= p \circ f \circ p' \in \mathcal{G}_{5m}$ is conjugated to an element $g'$ of $H_{2^{t}}\cap \mathcal G_{15m}$. By Proposition \ref{th2}, $p$ and $p'$ are commutators then $g\in [\mathcal G, \mathcal G]$. Moreover, $[\mathcal G, \mathcal G]$ is normal so  $g'\in H_{2^{t}}\cap [\mathcal G,\mathcal G]$.

Therefore, according to  Proposition \ref{Vasit}
$$c_{\mathcal G} (f) \leq c_{\mathcal G} (g) + 2= c_{\mathcal G} (g') + 2 < \frac{1}{2^t} c_{  H_{2^{t}}} (g') + 5.$$
As $c_{H_{2^{t}}} (g') \leq C_{15m}(\mathcal G )$, one has for any $t\in \mathbb N^*$ 
$$c_{\mathcal G}(f) < \frac{C_{15m}(\mathcal G )}{2^t} + 5. $$
Choosing $t$ large enough, we get  $\displaystyle c_{\mathcal G}(f) \leq  5.$\end{proof}

\subsection{The commutator length is bounded when prescribing the arithmetic.} \ 

Let $p\in \mathbb N^*$ and $\alpha=(\alpha_1, \cdots, \alpha_p) \in [0,1)^p$ such that $\alpha_1\notin \mathbb Q$.
 
\subsubsection{Background material.} \ 

We denote by $\Delta_{\alpha}$ the abelian subgroup of $\mathbb R$ generated by $\alpha_1,\cdots,\alpha_p$ and $1$.

Note that the condition $\alpha_1\notin \mathbb Q$ insures that $\Delta_{\alpha}$ is dense in $[0,1)$.

Let $J$ be a half-open interval with endpoints in $\Delta_\alpha$. 

\begin{defi} \ 

We set $\Gamma_{\alpha}:=\{ g \in \mathcal{G} : BP(g) \subset \Delta_{\alpha}\}$ 
and $\Gamma_{\alpha} (J):=\{ g \in \mathcal{G}_J : BP(g) \subset \Delta_{\alpha}\}$.
\end{defi}

\smallskip
It is plain that any $g\in \mathcal{G}$ is either periodic or belongs to some $\Gamma_{\alpha}$.  Indeed, if $g$ is not periodic then its length vector $\lambda$  has at least one irrational coordinate and $\alpha$ is obtained from $\lambda$ by permutation.

Note that $\Gamma_{\alpha} (J)$ is the collection of all  maps $g\in \mathcal G_J $ whose extensions to $I$ by the identity map belong to $\Gamma_{\alpha}$. But $\Gamma_{\alpha} (J)$ does not coincide with the set obtained by conjugating $\Gamma_\alpha$ through the homothecy that carries $J$ into $I$. For $J=[c,d)$, the last set is $\{g\in \mathcal G_J : BP(g)\subset c +  \frac{\Delta_{\alpha}}{d-c}\}$. For this reason, the properties of $\Gamma_{\alpha} (J)$ are not direct consequences of the ones of $\Gamma_{\alpha}$.

\begin{propri} \label{CJN} Let $g\in\Gamma_{\alpha}(J)$ and $I_i$ be the continuity intervals of $g$.
\begin{enumerate}
\item The lengths of the $I_i$ and the translations of $g$ are in $\Delta_{\alpha}$.

\smallskip

\item $\Gamma_{\alpha}(J)$ is a subgroup of $\mathcal G_J$.

\smallskip

\item The endpoints of the connected components of $Fix(g)$ belong to $\Delta_{\alpha}$.
\end{enumerate} \end{propri}

\begin{proof}  
 
\noindent \textit{Item (1).} The endpoints of the $I_i$ are the discontinuity points of $g$ and the left endpoint of $J$. Therefore if $g\in\Gamma_{\alpha}(J)$ any length $\lambda_i=\vert I_i \vert$ belongs to $\Delta_{\alpha}$. The translations of $g$ also belong to $\Delta_{\alpha}$, as linear combinations of the $\lambda_i$'s with coefficients in $\{-1,0,1\}$.

\noindent \textit{Item (2).} According to Item (1), any $f\in \Gamma_{\alpha}(J)$  preserves $\Delta_{\alpha}$. Therefore the relations $ BP(f^{-1}) = f (BP(f))$ and  $BP(f_1 \circ f_2) \subset BP(f_2) \cup f_2^{-1} (BP(f_1))$ imply that $\Gamma_\alpha(J)$ is stable by taking inverse and composite.

\noindent \textit{Item (3).} From the definition of $Fix(g)$, it follows that every endpoint of a connected component of $Fix(g)$ is a discontinuity point of $g$. 
\end{proof}

Before stating our last theorem, we give
\begin{defi} We say that $\mathcal G$ has \textbf{partial uniform perfectness} if for any  $p\in \mathbb N^*$ and $\alpha \in [0,1)^p$ it holds that $C_\alpha(\mathcal G) := sup \{  c_{\mathcal G} (g ) \ , \ g \in [\mathcal G, \mathcal G] \cap \Gamma_{\alpha}\}$ is finite. \end{defi}

\begin{theo}\label{th4} 
If $\mathcal G$ has partial uniform perfectness then $\mathcal G$ is uniformly perfect.
\end{theo}
A consequence of Theorems \ref{thGm} and \ref{th4} is
\begin{coro}
If $\mathcal G$ has partial uniform perfectness  then $c(\mathcal G)\leq 5$.
\end{coro}

The main tool for the proof of Theorem \ref{th4} is 

\begin{prop} \label{invBIET} Let $n$ be a positive integer and set $s_n=[\frac{ln(n)} { ln(1.25)}]+1$. Let $f \in \Gamma_{\alpha}$.  

Then there exist $g_n \in H_n \cap \Gamma_{\alpha}$, a map $h \in \Gamma_{\alpha}$ and $s_n$ involutions $i_j \in\Gamma_{\alpha}$, $j=1,2,...s_n$ such that $f=i_{1} \circ ...  \circ i_{s_n} \circ (h \circ g _n\circ h^{-1})$.
\end{prop}

For proving this proposition we use the following

\begin{lemm}\label{Lem4.2} 
Let $\epsilon \in (0, 1)$ and $J$ be a half-open interval with endpoints in $\Delta_\alpha$. 

If $f \in \Gamma_{\alpha}(J)$ then there exists an involution $i\in \Gamma_{\alpha}(J)$ such that 
$$\vert Fix(i \circ f) \vert \geq  \vert Fix(f) \vert+ \frac{\vert J \vert -\vert Fix f \vert}{5} \ (1-\epsilon).$$
\end{lemm}

\begin{proof}  Let $f \in \Gamma_{\alpha}(J)$ and $BP(f)=\{0=a_1,\cdots,a_m\}$, we set $a_{m+1}=1$.

\noindent \textbf{Case 1:} $Fix(f) =\emptyset$.

Fix $\delta\in \Delta_{\alpha}$ such that $0<\delta < \min\{\frac{\vert J \vert \epsilon}{m}\ ; \ \vert f(x)-x\vert, x\in I\}$.

For every $j\in\{1,\cdots, m-1\}$, we consider the unique integer $n_j$ such that $(n_j-1)\delta \leq \vert [a_j, a_{j+1})\vert < n_j \delta$. It holds that  $[a_j, a_{j+1})$ is  the union of $(n_j-1)$ intervals of length $\delta$ and an eventually empty interval $F_j$ of length less than $\delta$.

\medskip

Therefore $J$ can be decomposed as a finite union of pairwise disjoint half-open intervals $I_1, \cdots ,I_n$ and $F_1, \cdots ,F_{m}$ such that 
\begin{itemize}
\item  $f$ is continuous on these intervals, 
\item  $\vert I_j\vert=\delta$ for $j=1,...,n$ and $\vert F_{k} \vert<\delta $ for  $k=1,...,m$.
\end{itemize}
It follows that $n\delta + \sum\vert F_{k} \vert=\vert J \vert$.

Since for any $x$ it holds that $\vert f(x)-x\vert >\delta$, one has $f(I_j) \cap I_j =\emptyset$.
Therefore there exists an involution $i_1$ of support $I_1 \cup f(I_1)$ such that $i_1\vert_{I_1}=f\vert_{I_1}$ and then $i_1\circ f\vert_{I_1}= Id\vert_{I_1}$.

\smallskip

Now, we want to construct a similar involution $i_2$ on a second interval $I_{p_2}$ so that $i_1$ and $i_2$ have disjoint supports. This can be done if and only if $(I_{p_2}\cup f(I_{p_2}))\cap (I_1 \cup f(I_1))=\emptyset$. This means that 
$$I_{p_2} \subset I \setminus (I_1 \cup f(I_1)\cup f^{-1}(I_1)).$$

Consequently, such an interval $I_{p_2}$ and its corresponding involution $i_2$ with support $I_{p_2}\cup f(I_{p_2})$ exist provided that
$$(1) \ \ \ I' \setminus (I_1 \cup D(f(I_1)\cup f^{-1}(I_1))\not= \emptyset,$$
where $I'= I\setminus \cup F_k$ and $\displaystyle D(K) = \bigcup_{\{k \ \vert \  K\cap I_k \not= \emptyset\}}I_k$.

\smallskip

As any half-open interval of length $\delta$ meets at most two intervals $I_k$, the condition (1) means that $n>5$.

\medskip

By induction, we can define $s$ involutions $i_j$ with disjoint supports $I_{p_j}\cup f(I_{p_j})$ provided that 
$$I' \setminus (I_1\cup \cdots I_{p_{s-1}}  \cup D(f(I_1\cup \cdots I_{p_{s-1}}  )\cup f^{-1}(I_1\cup \cdots I_{p_{s-1}} ))\not= \emptyset.$$ That is $n > 5(s-1).$

\smallskip

Let $s$ be the largest integer such that $n >5(s-1)$, we can construct the involutions $i_j$, $j=1, ... , s$ but $n\leq 5s.$

\smallskip

By the definition of $i_j$, the map $g=i_s \cdots i_1 \circ f$ satisfies $$g\vert_{I_1 \cup I_{p_2} \cup \cdots \cup I_{p_s}}= Id\vert_{I_1 \cup I_{p_2} \cup \cdots \cup I_{p_s}}, {\text { then }}$$ 
$$\vert Fix(g)\vert \geq \sum_{j=1}^{s} \vert I_{p_j} \vert = s. \delta = \frac{s}{n}(\vert J \vert -\sum\vert F_k \vert)\geq \frac{1}{5} (\vert J \vert-\sum\vert F_k \vert)$$ $$ \geq \frac{1}{5}(\vert J \vert-m\delta) \geq\frac{1}{5}(\vert J \vert-\vert J \vert\epsilon).$$ 

In conclusion, since $i_j$ have disjoint supports, the map $i=i_s \cdots i_1 $ is an involution and $\vert Fix(i\circ f) \vert \geq \frac{\vert J \vert}{5}(1-\epsilon)$, this is the desired conclusion for $\vert Fix f \vert =0$.

\smallskip
It remains to prove that $i\in\Gamma_{\alpha}$. Since the $a_i$ and $\delta$ are in $\Delta_{\alpha}$, the endpoints of $I_i$ and $f(I_i)$  also belong to  $\Delta_{\alpha}$. Combining this with the fact that the discontinuity points of the involutions $i_j$ are endpoints of $I_i$ or $f(I_i)$, we get that $BP(i_j)\subset \Delta_{\alpha}$, for $j=1,...,s$. Therefore, by definition, the maps $i_j \in \Gamma_{\alpha}$ then so does $i$.

\bigskip

\noindent \textbf{Case 2:} $Fix(f) \not=\emptyset$.  We set $J=[c,d)$.

\smallskip

As the endpoints of the connected components of $Fix(f)$ belong to $\Delta_{\alpha}$, it holds that $a= \vert Fix(f)\vert \in\Delta_{\alpha}$. Therefore a slight adaptation of Lemma  \ref{Lemsup} to $f\in \Gamma_{\alpha}(J)$, shows that there exists $h \in\Gamma_{\alpha}(J)$ such that $Fix(h \circ f\circ h^{-1})=[c,c+a)$.

\smallskip

Let $f_1\in \Gamma_{\alpha}([c+a,d))$ be the restriction of $h \circ f\circ h^{-1}$ to $[c+a,d)$. By construction, $Fix(f_1) =\emptyset$ hence Case 1 applies to $f_1$ and provides an involution $j_1\in \Gamma_{\alpha}([c+a,d))$ such that $$\vert Fix(j_1\circ f_1) \vert \geq \frac{\vert J \vert-a}{5}(1-\epsilon).$$

Let $j\in \Gamma_{\alpha}(J)$ be the involution of $J$ defined by $j(x) =j_1(x)$ if  $x\in [c+a,d)$ and $j(x)=x$ if $x\in [c,c+a)$. We have $$\vert Fix(j\circ h \circ f\circ h^{-1}) \vert \geq  \vert [c,c+a) \vert + \vert Fix(j_1\circ f_1) \vert  \geq  a + \frac{\vert J \vert-a}{5}(1-\epsilon).$$ 
 
In addition, as $h$ preserves lengths, we have $$\vert Fix( j \circ h \circ f\circ h^{-1}) \vert =\vert Fix (h^{-1} \circ  (j \circ h \circ f\circ h^{-1}) \circ h ) \vert= \vert Fix ( (h^{-1} \circ j \circ h) \circ f)\vert.$$

Setting $i=h^{-1} \circ j \circ h$, we get $$\vert Fix(i\circ f)\vert\geq a+ \frac{\vert J \vert-a}{5}(1-\epsilon),$$ which completes the proof. \end{proof}

\bigskip

We turn now on to the proof of Proposition \ref{invBIET}.
\medskip

Let $\epsilon\in (0, 1)$ small enough and such that $\frac{1}{5}(1-\epsilon)\in \Delta_{\alpha}$. Consider $f \in \Gamma_{\alpha}$ and set $L_0=\vert Fix(f) \vert$.

\smallskip

Applying Lemma \ref{Lem4.2} to $f$, there exists an involution $i_1\in \Gamma_{\alpha}$ such that
$$\vert Fix(i_1\circ f) \vert \geq L_0+\frac{1-L_0}{5}( 1-\epsilon) =\phi(L_0) :=L_1,$$ \noindent \text{where} $\phi(x) := x+\frac{1-x}{5}( 1-\epsilon) = \frac{4+\epsilon}{5}( x-1) +1$ is a direct affine map whose fixed point is $1$.

\medskip

We now apply this argument again, with $f$ replaced by $i_1\circ f$, to obtain an involution $i_2\in \Gamma_{\alpha}$ such that$$\vert Fix(i_2 \circ i_1\circ f) \vert \geq\vert Fix(i_1\circ f)\vert +\frac{1-\vert Fix(i_1\circ f)\vert}{5}( 1-\epsilon) =\phi(\vert Fix(i_1\circ f)\vert)\geq \phi(L_1)=\phi^2(L_0).$$

Repeating this process $s$ times, we get $s$ involutions $i_k\in \Gamma_{\alpha}$ such that
$$\vert Fix(i_s\circ \cdots \circ i_1\circ f) \vert \geq \phi^s(L_0).$$

We now prove that $\phi^s(L_0)\geq 1-\frac{1}{n}$ provided that $s \geq s_n = \left[ \frac{ln (n)}{ln (1.25) }\right] +1$.  In order to get this inequality, we are looking for integers $s$ such that 
$$\phi^s (L_0) =  (\frac{4+\epsilon}{5} )^s( L_0-1) +1 \geq 1-\frac{1}{n} ,$$ that is 
$$ - (\frac{4+\epsilon}{5} )^s( 1-L_0) \geq -\frac{1}{n}.$$
Using that $L_0\geq 0$, it suffices to determine $s$ satisfying  
$$ (\frac{4+\epsilon}{5} )^s\leq  \frac{1}{n} $$
$$ s . ln (\frac{4+\epsilon}{5} )  \leq ln (\frac{1}{n} )= -ln(n) $$
$$ s \geq \frac { ln(n) }{ ln (\frac{5}{4+\epsilon} )} $$ 
Therefore, we can take $s  = \left[\frac { ln(n) }{ ln (\frac{5}{4+\epsilon} )} \right] +1$. In addition, as $\frac { ln(n) }{ ln (\frac{5}{4} )}\notin \mathbb N$, we have $ \left[\frac { ln(n) }{ ln (\frac{5}{4+\epsilon} )} \right] = \left[\frac { ln(n) }{ ln (\frac{5}{4} )} \right] $ for $\epsilon>0$ small enough.

Finally  $i_s\circ \cdots \circ i_1\circ f \in \Gamma_{\alpha}$ has a fixed point set of length at least $1-\frac {1}{n}$ so it is conjugated  to an element of $H_n$ by some $h\in \Gamma_{\alpha}$, by a slight adaptation of Lemma  \ref{Lemsup} to $f\in \Gamma_{\alpha}$. 

\subsubsection{Proof of Theorem \ref{th4}.}

We consider $g_1\in \Gamma_{\alpha} \cap [\mathcal G,\mathcal G]$ that realizes $C_{\alpha}(\mathcal G)$. By Proposition \ref{invBIET} with $n=2$ and thus $s_2 = \left[ \frac{ln (2)}{ln (1.25) } \right] +1 =4$, there exist  $g_2\in H_2\cap \Gamma_{\alpha}$, $h\in  \Gamma_{\alpha}$ and four involutions $i_1, \cdots, i_4$ such that $$g_1 =i_{1} \circ i_{2} \circ i_{3}\circ  i_{4}\circ (h g_2 h^{-1}) .$$
We can now estimate $C_{\alpha}(\mathcal G)$. By Proposition \ref{th2} and the normality of $[\mathcal G,\mathcal G]$,
$$C_{\alpha}(\mathcal G)= c_{\mathcal G} (g_1) \leq 4 + c_{\mathcal G} (g_2). $$
\text{According to Lemma \ref{VaserOpti}, we have}
$$c_{\mathcal G}(g_2) \leq \frac{1}{2} c_{H_2}( g_2) +\frac{3}{2}.$$
Using Remark \ref{rem1}, the group $H_2$ inherits the partial uniform perfectness of $\mathcal G$  and this implies that  for any  $g \in [H_2, H_2] \cap \Gamma_{\frac{\alpha}{2}}$, one has  $c_{ H_2} (g)\leq C_{\alpha}(\mathcal G)$. In particular since $\Gamma_{\alpha}$ is a subgroup of $\Gamma_{\frac{\alpha}{2}}$, we have 
 $c_{H_2}( g_2) \leq C_{\alpha}(\mathcal G)$.  Hence, 
$$C_{\alpha}(\mathcal G)\leq  4 + \frac{1}{2} C_{\alpha}(\mathcal G)+ \frac{3}{2}$$
$$ \frac{1}{2} C_{\alpha}(\mathcal G)\leq  4 + \frac{3}{2}= \frac{11}{2}$$
$$ C_{\alpha}(\mathcal G) \leq  11.$$
Finally, since any IET $g$ is either periodic or it belongs to some $\Gamma_{\alpha}$, we get that $c_{\mathcal G}(g)\leq 11$, for all $g\in [\mathcal G,\mathcal G]$. It means that $\mathcal G$ is uniformly perfect. \hfill $\square$

\begin{appendix}
\section{Note on Involution length}

\bigskip

In this appendix, we prove 
\begin{theo} \label{Th1} \ 

Any element of $\ietf$ is the product of at most $6$ strongly reversible elements in $\ietf$.
\end{theo}

\begin{coro} \label{Cor1}  Any element of $\ietf$  is the product of at most $12$ involutions in $\ietf$. \end{coro}

\subsection{Preliminairies.} 

\subsubsection{Strong reversibility in $\iet$ and $\ietf$} We list some basic properties on strongly reversible maps.
\begin{propri} \label{reve} \ 
\begin{enumerate}
\item Basic general properties. {Let $G$ be a group and $f_1$, $f_2$, $i_1$, $i_2$, $f$ and $R$ be elements of $G$.}
\begin{enumerate}
\item The identity and involutions are strongly reversible and an element is strongly reversible if and only if it is the product of at most $2$ involutions.
\item The conjugate of a strongly reversible element is strongly reversible. 
\item Let $f_k$ be strongly reversible by $i_k$, for $k=1,2$. If $\{i_1,f_1\}$ and $\{i_2,f_2\}$ are commuting then the product $f_1f_2$ is strongly reversible by $i_2 i_1$.
\item {If $R$ is an involution, then $f R f^{-1} R^{-1}$ is strongly reversible by $R$.}
\end{enumerate}
\item Specific properties.
\begin{enumerate}
\item Restricted rotations are strongly reversible in $\ietf$.
\item Any periodic element of $\iet$ is strongly reversible in $\iet$. 
\end{enumerate}
\end{enumerate}
\end{propri}

We let the reader check the basic general properties and we only prove the specific ones.

The rotation $R_{\alpha}$ of $[0,1)$ given by $R_{\alpha}(x) = x+\alpha \ (\mo 1 )$ is reversible by the 
symmetry $\I$ defined as $\I(x) =1-x$. Therefore, a restricted rotation of support $J$ is reversible by $\mathcal I_J$, the symmetry of $J$.

\hskip 2 truecm
\includegraphics[width=8cm, height=4cm]{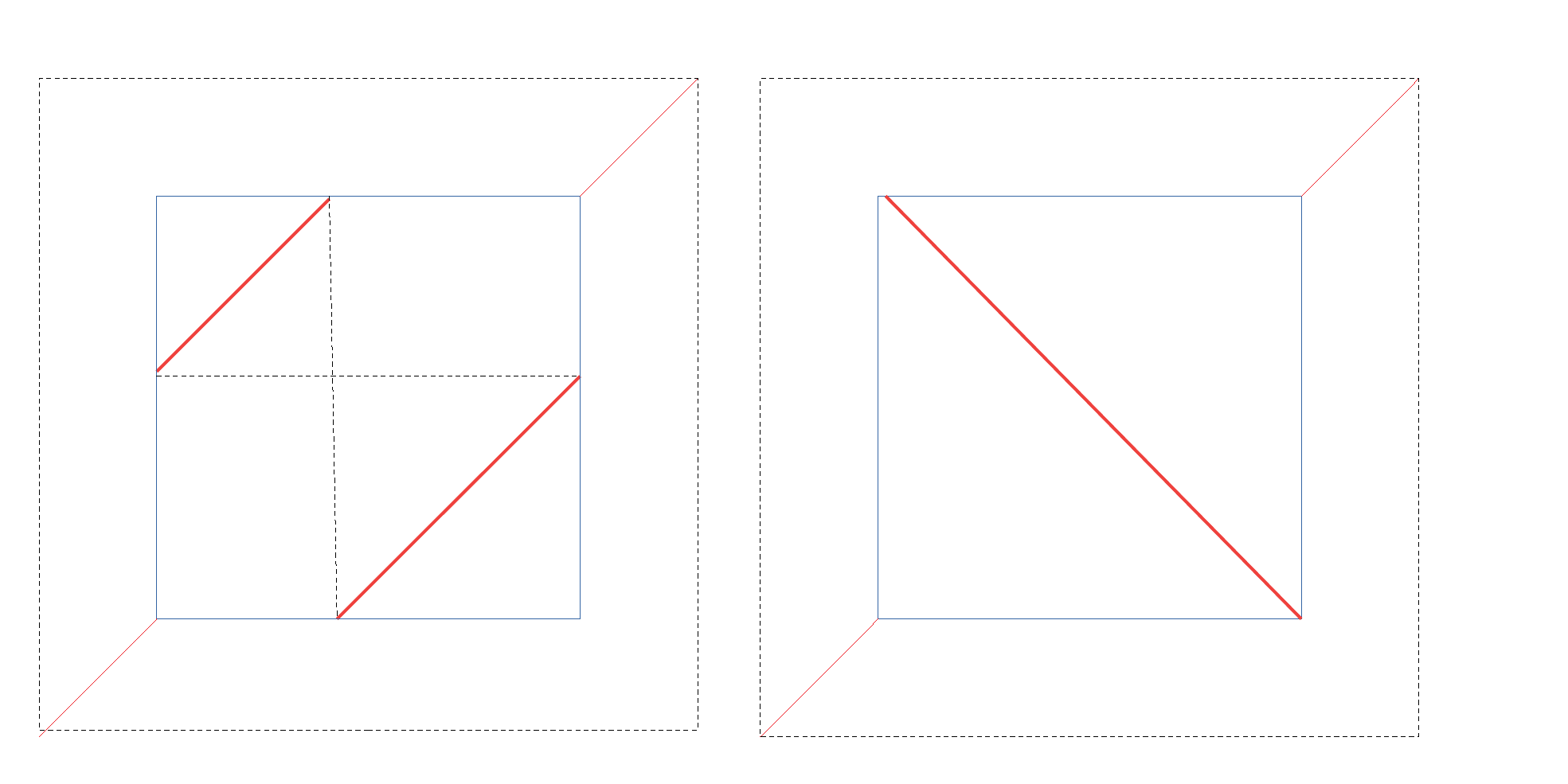} 

\vskip -0,4truecm \ \hskip 3,5truecm  $R_J$ \hskip 3truecm $\mathcal I_J$

It was already proved in \cite{GLIET} that any periodic \ie is strongly reversible. Here, we give a more direct proof.

Let $g$ be a periodic \ie and $p\in \mathbb N^*$ such that $g^p=\id$.  We consider the partition $\p$ of the interval $[0,1)$ into  half-open intervals $J_j$ defined by the discontinuity points of the maps $g$, $g^2$, ..., $g^{p-1}$. In particular, the maps $g$, $g^2$, ..., $g^{p-1}$ are continuous on each $J_j$ and there exists a minimal integer $p_j\leq p$ such that $g^{p_j}_{\vert J_j}=\id$. The partition $\p$ can be decomposed as the union of $\p_k=\{ J_k, ..., g^{p_k-1}(J_k) \}$ for $k\in\rl$,  {where $\rl$ is an index set for representatives of the $g$-orbits of the $J_j$'s}. 

Eventually conjugating $g$ by an \ie, we can suppose that  for all $k\in \rl$, the intervals $J_k$ ,..., $g^{p_k -1}(J_k)$ are consecutive so that their union is an half-open interval denoted by $I_k$. 

Let $\mathcal S_{p_i}$ denote the $p_i$-symmetric group and let $\pi \in \mathcal S_{p_i}$ be the cycle $(1,2,\cdots, p_i)$ that is $\pi(t) =t+1 \ (\mo p_i)$. According to \cite{FS} (Proposition 3.4 page 41), there exists $\tau \in \mathcal S_{p_i}$ such that $\tau$ has order $2$ and it satisfies $\tau^{-1} \pi \tau (k)=\pi^{-1}$.

We consider the \ie $h$, that is defined on each $I_i$ by $h$ is continuous on $g^{t}(J_i)$, $t=0,...,p_i-1$ and $h(g^{t}(J_i)) =g^{\tau(t)}(J_i)$. One can easily check that $h$ is an involution and $h^{-1}\circ g \circ h$ is continuous on $g^{t}(J_i)$ and $h^{-1}\circ g \circ h(g^{t}(J_i)) = h^{-1}\left(g^{\pi \tau(t)}(J_i)\right)=g^{\tau^{-1} \pi \tau(t)}(J_i)=g^{\pi^{-1}(t)}(J_i)=g^{-1}(g^t(J_i))$. Therefore $h^{-1}\circ g \circ h=g^{-1}$, meaning that $g$ is strongly reversible in $\iet$. 

\medskip

Lemma 3.2 and \ref{reve} {(2 a)} imply  

\begin{prop}\label{prop1} \

Any element of $\iet_m$ is the product of at most $m-1$ strongly reversible elements in $\ietf$.
\end{prop}

\subsubsection{The Dennis-Vasserstein's criterion for strongly reversible maps.}

\begin{defi} \label{srevlength}
Let $G$ be a simple group and $g\in G$, the strongly reversible length of $g$, denoted by  $\boldsymbol{\rl_{G} (g)}$, is the least number $r$ such that $g$ is a product of $r$ strongly reversible elements in $G$.
\end{defi}

\begin{prop} \label{DVrevit}
Let $n\in \mathbb N^*$. If $g\in {H_{2^n}}$ then  $\rl_{\ietf} (g) < \frac{1}{2^n} \rl_{H_{2^{n}}} (g) + 3.$
\end{prop}

\begin{lemm}\label{DVRevers} If $g\in H_2$ is the product of $4$ strongly reversible elements in $H_2$ then
$g$ is the product of $3$ strongly reversible elements in $\ietf$.
\end{lemm}

\begin{proof}
Let $g= \rho_1 \rho_2 \rho_3 \rho_4$ with $\rho_i$ strongly reversible elements in $H_2$. Denote by $\rho'_i$ the map $R_{\frac{1}{2}} \rho_i R_{\frac{1}{2}}$ which commutes with the $\rho_j$ and is strongly reversible in $ R_{\frac{1}{2}} H_2  R_{\frac{1}{2}}$. We have
$$g= \rho_1 \rho_2 \rho_3 \rho_4 \ (\rho'_3 \rho'_4) \ (\rho'_3 \rho'_4)^{-1}$$
$$g= \underbrace{( \rho_1  \rho'_3 )}_{C_1} \ \underbrace{(\rho_2 \rho'_4)}_{C_2} \ \underbrace{( \rho_3 \rho_4)\bigl((\rho_3 \rho_4)^{-1}\bigr)'}_{C_3}.$$
According to Properties \ref{reve} $(1)$ $c$ and $d$, the $C_i$ are strongly reversible and the conclusion follows.
\end{proof}

More generally, if $g$ is the product of $2p$ strongly reversible elements in $H_2$ then 
$$g= (\rho_1 ...\rho_{2p}) = (\rho_1 ... \rho_{2p}) \ ( \rho'_{p+1}... \rho'_{2p}) \ (\rho'_{p+1}... \rho'_{2p})^{-1}$$ 
 $$g= \underbrace{(\rho_1  \rho'_{p+1})}_{C_1} \ ...\ \underbrace{(\rho_p \rho'_{2p})}_{C_p} \ \  \underbrace{(\rho_{p+1}... \rho_{2p}) \ ( \rho'_{p+1}... \rho'_{2p})^{-1}}_{C_{p+1}}.$$ 
Therefore $\rl_{\ietf}(g) \leq \frac{1}{2} (\rl_{H_2}(g) + 3)$.

\medskip

Starting from $g\in  {H_{2^n}}$ and iterating this process for the groups $H_{2^t}$ and $H_{2^{t+1}}$, for $t<n$, we get 
$$\rl_{\ietf} (g)  \leq \frac{1}{2} (\rl_{H_2}(g) + 3) \leq \frac{1}{4} (\rl_{H_4}(g) + 9)\leq ... \leq 
\frac{1}{2^n}\left(\rl_{H_{2^{n}}} (g)  + 3(2^n-1)\right).$$
Finally, it holds that
$$\rl_{\ietf} (g) <\frac{1}{2^n} \rl_{H_{2^{n}}} (g) + 3.$$

\subsection{Proof of  Theorem \ref{Th1}}
Let $f\in \iet$ and  $m= \# \disc(f)$. By Proposition 5.1, for all $n\in \mathbb N^*$, there exist $p$, $p'$ periodic, $g\in H_{2^n} \cap \iet_{15m}$ and $h\in \iet$ such that $f=p \circ(h^{-1} \circ g \circ h) \circ p'$. Then 
$$\rl_{\ietf} (f) \ \leq  \rl_{\ietf}(p)+\rl_{\ietf}(h^{-1} \circ g \circ h)+\rl_{\ietf} (p').$$
By Properties \ref{reve}, the maps $p$ and $p'$ are strongly reversible and then
$$\rl_{\ietf} (f) \ \leq \ 2  \ + \rl_{\ietf} (g).$$
Using  Proposition \ref{DVrevit}, we get 
$$\rl_{\ietf} (f) \ < \ 2 \ +  \frac{1}{2^n} \rl_{H_{2^n}} (g) + 3.$$
As $g\in \iet_{15m}$, {Remark \ref{rem1} and Proposition \ref{prop1} imply} that
$$\rl_{\ietf} (f) \  < \  5 \ + \ \frac{15m}{2^n}.$$
Finally, letting $n\rightarrow +\infty$, we get 
$$\rl_{\overline{\mathcal G}} (f) \ \leq \ 5.$$

In Conclusion: let $\overline{f}\in \ietf$. By Lemma \ref{invRR} (1), $\overline{f}= i\circ f$ with $i$ an involution and $f\in \iet_m$ thereby $\rl_{\ietf} (f) \leq 1+5=6.$

\subsection{Proof of Corollary \ref{Cor1}}

Let $f\in \ietf$, by Theorem \ref{Th1}, $f$ can be written as a product of $6$ strongly reversible elements of $\ietf$. By Properties \ref{reve}, these elements are product of $2$ involutions then $f$ is the product of $12$ involutions.

\end{appendix}

\bibliographystyle{alpha} 
\bibliography{SrefSimpleSinNombre} 
\end{document}